\documentclass{article}
\usepackage{amsmath}
\usepackage{amsfonts}

\newtheorem{theorem}{Theorem}

\newtheorem{lemma}[theorem]{Lemma}

\newtheorem{proposition}[theorem]{Proposition}
\newtheorem{remark}[theorem]{Remark}

\newenvironment{proof}[1][Proof]{\noindent\textbf{#1.} }{\ \rule{0.5em}{0.5em}}
\input{tcilatex}

\def\R{{\mathbb R}}

\def\E{{\mathbb E}}
\def\P{{\mathbb P}}
\def\S{{\mathbb S}}

\begin{document}

\title{Strong Local Nondeterminism and Exact Modulus of Continuity for
Spherical Gaussian Fields}
\author{Xiaohong Lan \\
%EndAName
School of Mathematical Sciences \\ University of Science and Technology of China
\and Domenico Marinucci \thanks{Corresponding author} \\
%EndAName
Department of Mathematics \\ University of Rome Tor Vergata \and Yimin Xiao \\
%EndAName
Department of Statistics and Probability\\ Michigan State University}
\maketitle

\begin{abstract}
In this paper, we are concerned with sample path properties of  isotropic spherical
Gaussian fields on $\S^2$. In particular, we establish the property of strong local
nondeterminism of an  isotropic spherical Gaussian field based on the high-frequency
behaviour of its angular power spectrum; we then exploit this result to establish
an exact uniform modulus of continuity for its sample paths. We also discuss the
range of values of the spectral index for which the sample functions exhibit fractal
or smooth behaviour.
\end{abstract}



\section{\protect\bigskip Introduction and Overview}

\subsection{Motivations}

The analysis of sample path properties of random fields has been
considered by many authors, see, %the recent books \cite{RFG, MRbook},
for instance, \cite{DMX15, FalconerXiao, LiWangXiao2015,
LuanXiao, MeerschaertXiao, Pitt, Ta95, Ta98, X07, Xiao09} and their
combined references. These papers have covered a wide variety of
circumstances, including scalar and vector valued random fields,
isotropic and anisotropic behaviour, analytic and geometric properties.
%and various forms of non-Gaussianity,
%see i.e., \cite{XiaoReview2013} for a review.
The parameter space of the random fields in these references, however,
has been typically considered to be Euclidean, i.e., %(compact) domains in
$\mathbb{R}^{k},$ for $k\geq 1.$

From the point of view of applications, however, there is of course a lot of
interest in investigating  sample path properties of random fields
defined on manifolds. In particular, we shall focus here on isotropic
random fields defined on the unit sphere $\S^{2};$ these fields have
considerable mathematical interest by themselves, and arise very naturally
in a number of scientific areas, i.e., geophysics, astrophysics and
cosmology, athmospheric sciences, image analysis, to name only a few, see
\cite{MPbook} for a systematic account. To the best of our knowledge, very
little is currently known on the sample path properties of these fields,
even under Gaussianity and Isotropy assumptions; the only currently
available references seem to be \cite{lang2, lang1}, which investigate
differentiability and H\"older continuity properties of the sample functions
in terms of the so-called spectral index, to be defined below.

Our aim in this paper is to pursue this line of investigation further and to
provide two main results. The first of these results is to establish a property of
strong local nondeterminism for a large class of  isotropic spherical Gaussian
fields. In the Euclidean setting, the notion of strong local nondeterminism has
played a pivotal role to establish a number of characterizations for sample
trajectories, see again \cite{Pitt, Ta95, Ta98, X07, Xiao09, XiaoReview2013}
for more discussions and review of recent papers; we thus believe that our result
will open a way for similar developments in the area of spherical Gaussian
fields. In particular, by exploiting this property, we are able to establish
our second main result, i.e. the exact uniform modulus of continuity
for isotropic spherical Gaussian fields. The exact form of the scaling depends
in a very explicit way on the behaviour of the angular power spectrum (to be
recalled below) of the field, and we can hence identify the class of models that
lead to fractal properties. In oder to state more precisely these results, we
need to introduce however some more notation and background material,
which we do in the following subsection.

\subsection{Background and notation}

We start by recalling some background from \cite{MPbook} on second order
spherical random fields, by which we mean as usual measurable applications
$T:  \Omega \times \S^{2} \rightarrow \mathbb{R}$, where $\left\{ \Omega,
\Im , \mathbb P \right\} $ is some probability space, such that
 for all $x\in \S^{2}$,
\begin{equation*}
%\mathbb{E} (T(x,\omega ))= \int_{\Omega } T(x,\omega ) d{\mathbb P}(\omega )=0,
\mathbb E \big(T^{2}(x,\omega ) \big) =\int_{\Omega }T^{2}(x,\omega ) d\P(\omega )<\infty.
\end{equation*}%
Without loss of generality, in the sequel we shall always assume the field
to have zero-mean, $\mathbb{E} \big(T(x,\omega )\big)=0.$ Also, as usual,
by (strong) isotropy we mean that the random fields  $T= \{T(x), x\in \S^{2}\}$ and $T^g
= \{ T(gx ), x \in \S^{2}\}$ have the same law, for all rotations $g\in SO(3).$ $T$ is
called $2$-weakly isotropic if $\E \big(T(x) T(y)\big) = \E \big(T(gx) T(gy)\big)$ for all
$g\in SO(3).$

Given a $2$-weakly isotropic random field $T = \{T(x), x \in \S^2\}$,
the following spectral representation is well known to hold (cf. \cite[Theorem 5.13 ]{MPbook}):
\begin{equation} \label{Eq:T-rep1}
T(x;\omega ) = \sum_{\ell= 0}^\infty  \sum_{m = -\ell}^\ell a_{\ell m}(\omega )Y_{\ell m}(x),
\end{equation}
where $\{Y_{\ell m},\, \ell \ge 0; m = 0, \pm 1, \ldots, \pm \ell\}$ are the
spherical harmonic functions on $\S^2$ and $a_{\ell m} =
\int_{\S^2}T(x)\overline{Y_{\ell m}(x)}\, dx$.
The equality in \eqref{Eq:T-rep1} holds both in $L^{2}(\Omega )$ at every fixed $x,$
and in $L^{2}(\Omega \times \S^{2}),$ i.e.%
\begin{equation*}
\lim_{L\rightarrow \infty }\mathbb{E}\bigg[ T(x)-\sum_{\ell
}^{L}\sum_{m}a_{\ell m}(\omega )Y_{\ell m}(x)\bigg] ^{2}=0 ,
\end{equation*}%
and%
\begin{equation*}
\lim_{L\rightarrow \infty }\mathbb{E}\bigg[ \int_{\S^{2}}\bigg( T(x;\omega)-
\sum_{\ell }^{L}\sum_{m}a_{\ell m}(\omega )Y_{\ell m}(x)\bigg) ^{2}dx%
\bigg] =0.
\end{equation*}%
We recall that the finite-variance condition $\mathbb{E}\big(T^{2}(x) \big)<\infty$
under isotropy automatically entails the mean-square continuity; the spectral
representation hence follows without further assumptions, see \cite%
{MPbook,marpec2013}.

If $T = \{T(x), x \in \S^2\}$ is a Gaussian random field, then its strong isotropy
and $2$-weak isotropy are equivalent. %and  this paper, we are concerned with
%Gaussian isotropic spherical random fields.
The distribution of  an isotropic zero-mean Gaussian field $T = \{T(x), x\in \S^{2}\}$
is fully characterized by the covariance function $\mathbb{E}\big(T(x)T(y)\big).$
By a theorem of Schoenberg  \cite{schoenberg1942}, the latter can be expanded
as follows:%
\begin{equation} \label{Eq:Cov-T}
\mathbb{E} \big(T(x)T(y) \big)=\sum_{\ell =0}^{\infty }\frac{2\ell +1}{4\pi }C_{\ell}
P_{\ell }(\left\langle x,y\right\rangle );
\end{equation}%
here,  $P_0 \equiv 1$ and $P_{\ell }: [-1,1]\rightarrow \mathbb{R}$, for $\ell = 1, \, 2, ..., $
denote the Legendre polynomials, which satisfy the normalization condition $%
P_{\ell }(1)=1$ and can be recovered by Rodrigues' formula as
\begin{equation*}
P_{\ell }(t)=\frac{1}{2^{\ell }\ell !}\frac{d^{\ell }}{dt^{\ell }}%
(t^{2}-1)^{\ell },  \ \ \ \ell = 1,\, 2, ...
\end{equation*}%
On the other hand, the sequence $\left\{ C_{\ell },\  \ell
=0, 1, ... \right\} $ of nonnegative weights represents the so-called angular
power spectrum of the field, and the $\ell$'s are referred to as frequencies
(also labelled multipoles). In terms of the spectral representation, we
have the identification%
\begin{equation}\label{Def:Cl}
\mathbb{E}\big(a_{\ell m}\overline{a}_{\ell ^{\prime }m^{\prime }} \big)=C_{\ell
}\delta _{\ell }^{\ell ^{\prime }}\delta _{m}^{m^{\prime }},
\end{equation}%
so that the angular power spectrum provides the variance of the
(uncorrelated) Gaussian random coefficients $\left\{ a_{\ell m}, \ \ell
=0, 1,2,...; m=-\ell ,...,\ell \right\}.$  By standard Fourier arguments, the
small scale behaviour of the covariance is determined by the behavior of
the angular power spectrum at high frequencies;
namely, the behavior of  $C_\ell$ for as $\ell \to \infty.$

It is known that for $\ell = 0$, $Y_{00}(x)$ in \eqref{Eq:T-rep1} is a constant
function on $\S^2$, which does not affect the sample path regularity of $T(x)$.
Hence, for simplicity of notation, we will remove the term for $\ell = m =0$ from
\eqref{Eq:T-rep1}  and \eqref{Eq:Cov-T} (i.e., we consider $T(x) - a_{00}
Y_{00}(x)$) throughout the rest of this paper. Furthermore, we shall impose
the following condition on the behavior of the angular power spectrum, which
we consider in every respect as minimal.

\bigskip
%\begin{condition}
\noindent{{\bf Condition (A)}: The random field $T = \{T\left( x\right), x \in \S^2\} $
is zero-mean, Gaussian and isotropic, with angular power spectrum such that:
\begin{equation} \label{reg2}
%C_{\ell }(\vartheta )=C_{\ell }(G,\alpha _{0})
C_\ell =G\left( \ell \right) \ell^{-\alpha }>0 ,  \ \ \forall \, \ell =1,2, ...,
\end{equation}%
where $\alpha >2$ is a constant and, moreover,  there exists a finite constant
$c_{0}\ge 1,$ such that
\begin{equation*}
c_{0}^{-1}\leq G\left( \ell \right) \leq c_{0}.
\end{equation*}}

The assumption $\alpha >2$ is necessary to ensure that the field has finite
variance (recall the identity $\mathbb{E}\big(T^{2}(x)\big)=\sum_{\ell }\frac{2\ell +1}{%
4\pi }C_{\ell }$). On the other hand, we stress that we are imposing no
regularity condition on the function $G(\ell ),$ on the contrary of much of
the literature on spherical random fields, which typically requires $%
\lim_{\ell \rightarrow \infty }G(\ell )=const.$ or other forms of additional
regularity conditions (see i.e., \cite{bkmpAoS, spalan, MaVa,mayeli}).
We believe that Condition (A) covers the vast majority
of models which seems of interest from a theoretical or applied point of
view; for instance, it fits very well with the theoretical and observational
evidence on Cosmic Microwave Background radiation data (see \cite{dode2004,
Durrer, planck}), which has been one of the main motivating
areas for the analysis of spherical fields over the last decade. \ Most of
our results to follow will depend in a simple analytic way from the value of
the parameter $\alpha$, which we refer to  as the spectral index of $T$.

\subsection{Statement of the Main Results}

To introduce our first main result (on strong local nondeterminism),
we need first to introduce some more notation. In particular, for
$\alpha > 2$, let $\rho _{\alpha }:\mathbb{R}^{+}\rightarrow \mathbb{R}^{+}$
be the continuous function defined by
\begin{equation} \label{def:rho}
\rho _{\alpha }\left( t\right) =\left\{
\begin{array}{ll}
t^{(\alpha - 2)/{2}}, & \hbox{ if}\ 2<\alpha <4,\\
t\sqrt{| \log {t}|}, & \hbox{ if }\ \alpha =4,\\
t,  & \hbox{ if }\ \alpha >4\\
\end{array}%
\right.
\end{equation}%
and $\rho _{\alpha }(0)=0$ for all values of $\alpha .$  In the above
and  in the sequel, $\log x = \ln (x \vee e)$ for all $x > 0$. As we shall
show later, up to a constant factor the functions $\rho _{\alpha }$ can
be related to the canonical (Dudley) metric for the Gaussian processes
to  be investigated; it is important to note the explicit dependence on
the spectral index $\alpha.$ As usual, we take
\begin{equation*}
d_{\S^{2}}(x,y)=\arccos (\left\langle x,y\right\rangle )
\end{equation*}%
as the standard spherical (or geodesic) distance on $\S^{2}.$ The following
result establishes the property of strong local nondeterminism for spherical
Gaussian fields satisfying Condition (A) with $ 2 < \alpha < 4$.

\begin{theorem}
\label{ND}  Let $T = \{T(x), x \in \S^2\}$ be an isotropic Gaussian field that
satisfies Condition (A) with  $ 2 < \alpha  < 4$. There exist
positive and finite constants $c_{2}$ and $\varepsilon_0$ such that for all
integers $n\geq 1$ and all $x_{0},x_{1}, ... ,x_{n}\in \mathbb{S}^{2}$ with
$\min_{1\leq k\leq n} d_{\S^{2}} \big(x_{0},x_{k} \big)\le \varepsilon_0$
we have
\begin{equation}\label{SLND}
{\rm Var}\left( T\left( x_{0}\right) |T\left( x_{1}\right) ,...,T\left(
x_{n}\right) \right) \geq c_{2}\min_{1\leq k\leq n}\rho _{\alpha }\left(
d_{\S^{2}}(x_{0},x_{k})\right)^{2}.
\end{equation}
\end{theorem}

The proof of Theorem \ref{ND} is presented in Section \ref%
{LocalNondeterminism}. The argument does not seem to work for the critical
case of $\alpha =4$, we expect that (\ref{SLND}) still holds, but a new method
may be needed.

In the following we simply note how the strong local
nondeterminism property can be exploited to develop a number of nontrivial
characterizations for the sample path behaviour of spherical random fields.
Among these characterizations, in this paper we shall focus on the uniform
modulus of continuity, for which we shall establish the
following result, which significantly improves the H\"older continuity established
by Lang and Schwab \cite[Theorem 4.5]{lang1}.

\begin{theorem} \label{MC}
Let $T = \{T(x), x \in \S^2\}$ be an isotropic Gaussian field that
satisfies Condition (A).
\begin{itemize}
\item[(i).] If $2 < \alpha <4$, then  there exists a
positive and finite constant $K_1$ such that, with probability one
\begin{equation}\label{Eq:Umod}
\lim_{\varepsilon \rightarrow 0}\sup_{\substack{ x,y\in \S^{2},  \\ %
d_{\S^{2}}(x,y)<\varepsilon }}\frac{|T(x)-T(y)|}{\rho _{\alpha }\left(
d_{\S^{2}}(x,y)\right) \sqrt{\big| \log \rho _{\alpha }\left(
d_{\S^{2}}(x,y)\right)\big| }} = K_1.
\end{equation}
\item[(ii).] If $\alpha = 4$, then there exists a
positive and finite constant $K_2$ such that, with probability one
\begin{equation} \label{Eq:Umod2}
\lim_{\varepsilon \rightarrow 0}\sup_{\substack{ x,y\in \S^{2},  \\ %
d_{\S^{2}}(x,y)<\varepsilon }}\frac{|T(x)-T(y)|}{ d_{\S^{2}}(x,y) \big|\log
d_{\S^{2}}(x,y)\big| } \le K_2.
\end{equation}
\end{itemize}
\end{theorem}

The proof of Theorem \ref{MC} will be given in Section \ref{ModulusContinuity}.
In the following, we provide some remarks.

\begin{itemize}
\item In terms of the geodesic distance,
the results (\ref{Eq:Umod}) and (\ref{Eq:Umod2})  can be clearly written as
\begin{equation*}
\begin{split}
\lim_{\varepsilon \rightarrow 0}\sup_{\substack{ x,y\in \S^{2}, \\ %
d_{\S^{2}}(x,y)<\varepsilon }}\frac{|T(x)-T(y)|}{d_{\S^{2}}(x,y)^{(\alpha
-2)/2}\sqrt{\big|\log d_{\S^{2}}(x,y)\big|}} &= \sqrt{\frac{\alpha -2} 2}\, K_1,
 \hbox{ for } \, 2<\alpha <4.
%\lim_{\varepsilon \rightarrow 0}\sup_{\substack{ x, y\in \S^{2}, \\ %
%d_{\S^{2}}(x,y)<\varepsilon }}\frac{|T(x)-T(y)|}{d_{\S^{2}}(x,y) \big| \log
%d_{\S^{2}}(x,y) \big|} &\le \sqrt{2}\,K_2, \hbox { for } \alpha =4.
\end{split}
\end{equation*}%

\item It is important to note the fractal behaviour that occurs for $2<\alpha <4,$
when the modulus of continuity decays slower than linearly with respect to the
angular distance (hence the sample function $T(x)$ is nondifferentiable). We
note that this range of values of $\alpha$ is typical for many applied fields, for
instance for Cosmic Microwave Background data $\alpha $ is known to be
very close to $2,$ from theoretical arguments and from experimental data
(see e.g., \cite{planck}).

\item For the case of $\alpha = 4$, \eqref{Eq:Umod2} implies that the sample
function  $T(x)$ is almost Lipschitz. We believe the equality in \eqref{Eq:Umod2}
actually holds and the sample function presents subtle fractal properties.
However, we have not been able to prove these results, due to the unsolved
case in Theorem \ref{ND}.
%and a new method  for studying strong local nondeterminism in the case
%of $\alpha = 4$ will be needed (see Theorem \ref{C2} below for the case
%of $2 < \alpha < 4$).
\end{itemize}

Next we consider the case of  $\alpha >4$. Let  $k \ge  1$ be the unique integer
such that $2+ 2k < \alpha < 4 + 2k$. It follows from Lang and Schwab
\cite[Theorem 4.6]{lang1} that $T = \{T(x), x \in \S^2\}$ has a modification,
still denoted by $T$, such that its sample function is almost surely $k$-times
continuously differentiable. Moreover, the $k$-th (partial) derivatives of $T(x)$
are H\"older continuous on  $\S^2$ with exponent $\gamma <
\frac{\alpha - 2 }2 - k$.

%and we have shown that the sample function $T(x)$ is almost
%surely differentiable. However, the method of this paper is not sufficient to
%characterize the oscillations of its partial derivatives, due to their anisotropic
%nature.

In the following, we adapt the approach of Lang and Schwab \cite{lang1}
(see also \cite{lang2}) to study the regularity properties of higher-order
derivatives of $T$ based on pseudo-differential operators, as described in
the classical monograph \cite{Taylor}. In particular, for real $k\in \mathbb{R}$
introduce $(1-\Delta_{\S^{2}})^{k/2}$ as the pseudo-differential operator
whose action on functions $T(\cdot) \in L^{2}(\S^{2})$ is defined by
\begin{equation}\label{Tk}
(1-\Delta_{\S^{2}})^{k/2}T :=\sum_{\ell m}a_{\ell m} (1+\ell (\ell
+1))^{k/2}Y_{\ell m},
\end{equation}%
provided the right-hand side converges in $L^2(\Omega\times \S^2)$. In the above,
$\{a_{\ell m}\}$ is the same sequence of random variables as in \eqref{Eq:T-rep1},
and $\Delta _{\S^{2}}$ is the spherical Laplacian, also called Laplace-Beltrami
operator which, in spherical coordinates $(\vartheta ,\varphi ),$ is defined by
$0\leq \vartheta \leq \pi ,$ $0\leq \varphi <2\pi ,$
\begin{equation}\label{Lap}
\Delta_{\S^{2}}=\frac{1}{\sin \vartheta }\frac{\partial }{\partial \vartheta }%
\left\{ \sin \vartheta \frac{\partial }{\partial \vartheta }\right\} +\frac{1%
}{\sin ^{2}\vartheta }\frac{\partial ^{2}}{\partial \vartheta ^{2}}.
\end{equation}
Recall that for every $x \in \S^2$, it can be written as $x = (\sin \vartheta \cos \varphi,
\sin \vartheta \sin \varphi, \cos \vartheta)$. In this paper, with slight abuse of notation,
we always identify the Cartesian and angular coordinates of the point $x \in \S^2$.

It is shown in \cite[Chapter XI]{Taylor} that the Sobolev space $\mathcal{W}^{k,2}(\S^2)$
of  functions with square-integrable $k $-th derivatives can be viewed as the image of
$L^2(\S^2)$ under the operator $(1-\Delta_{\S^{2}})^{-k/2}$; this and related property
are exploited by Lang and Schwab \cite{lang1} to prove their Theorem 4.6 on regularity
of higher-order derivatives. More precisely,  consider the Gaussian random field
$T^{(k)}= \{T^{(k)}(x),  x \in \S^2\}$  defined by
\begin{equation*}
T^{(k)}:=(1-\Delta _{\S^{2}})^{k/2}T.
\end{equation*}%
Lang and Schwab \cite{lang1} study the almost-sure H\"older continuity of $T^{(k)}$.
We are able to improve their results by considering the exact modulus of continuity, for
which we provide the following result.

\begin{theorem}\label{MC-2}
If in Condition (A),  $2 + 2k < \alpha \le 4+2k$ for some integer $k \ge 1$, then
$T^{(k)} = \{T^{(k)}(x), x \in \S^2\}$  satisfies the following uniform modulus
of continuity:
\begin{itemize}
\item[(i).] If $2 +2k < \alpha <4 + 2k$, then  there exists a
positive and finite constant $K_3$ such that
\begin{equation*}
\lim_{\varepsilon \rightarrow 0}\sup_{\substack{ x,y\in \S^{2},  \\ %
d_{\S^{2}}(x,y)\le \varepsilon }}\frac{|T^{(k)}(x)-T^{(k)}(y)|}{\rho _{\alpha-2k }\left(
d_{\S^{2}}(x,y)\right) \sqrt{\big|\log \rho _{\alpha -2k}\left(
d_{\S^{2}}(x,y)\right) \big|}}=K_3, \quad \hbox{a.s.}
\end{equation*}
\item[(ii).] If $\alpha = 4 + 2k$, then there exists a
positive and finite constant $K_4$ such that
\begin{equation*} \label{Eq:Umod4}
\lim_{\varepsilon \rightarrow 0}\sup_{\substack{ x,y\in \S^{2},  \\ %
d_{\S^{2}}(x,y)<\varepsilon }}\frac{|T^{(k)}(x)-T^{(k)}(y)|}{ d_{\S^{2}}(x,y) \big|\log
d_{\S^{2}}(x,y)\big| } \le K_4, \quad \hbox{a.s.}
\end{equation*}
\end{itemize}
\end{theorem}

\subsection{Plan of the Paper}

The plan of the paper is as follows. In Section \ref{TechnicalTools} we
introduce some auxiliary tools that will be instrumental for our proofs
to follow; in particular, a careful analysis of the variogram/covariance
function on very small scales, and the construction of the so-called
spherical bump function, i.e. a compactly supported function on the
sphere satisfying some required smoothness conditions. The latter
construction builds upon ideas discussed by Geller and Mayeli
\cite{gm1,gm2} in the framework of spherical wavelets. In Section
\ref{LocalNondeterminism}, we exploit these results to establish the
property of strong local nondeterminism  for a large class of isotropic
spherical Gaussian fields. In Section  \ref{ModulusContinuity}, by applying
Gaussian techniques and strong local  nondeterminism we prove Theorem
\ref{MC} on the exact uniform modulus of continuity; while an extension
to higher-order derivatives is discussed in  Section \ref{HigherDerivatives}.
Some auxiliary results are collected in the Appendix.

\section{Technical Tools \label{TechnicalTools}}

\subsection{The Variogram}

It is well-known that, for the investigation of sample properties of Gaussian
field $T= \{T(x), \, x \in \S^2\}$, it is important to introduce the canonical metric%
\begin{equation*}
d_{T}(x,y)=\sqrt{\mathbb{E} \big(\left\vert T\left( x\right) -T\left(y\right)
\right\vert ^{2} \big)},
\end{equation*}%
see for instance \cite{RFG, MRbook} or any other monograph on the modern
theory of Gaussian processes. The square of the canonical metric is also
known as the variogram of $T$. Our first technical result is a careful investigation
on the behaviour of this metric for pairs of points that are very close in the
spherical distance $d_{\S^{2}}(\cdot,\cdot);$ more precisely, we have the
following upper and lower bounds, in terms of the function $\rho_\alpha$
which was introduced in (\ref{def:rho}).

\begin{lemma}\label{C1}
Under Condition (A), there exist constants $1 \le c_{1} < \infty$ and
$0<\varepsilon <1,$ such that for all $x, y \in \S^2$ with $d_{\S^{2}}(
x,y) \le \varepsilon,$ we have
\begin{equation}\label{Eq:Vario-bounds}
c_{1}^{-1}\rho _{\alpha }^2\left( d_{\S^{2}}\left(x,y\right)
\right) \leq d_{T}^{2}(x,y)\leq c_{1}\rho _{\alpha }^2\left(
d_{\S^{2}}\left( x,y\right) \right),
\end{equation}%
where $\rho _{\alpha }\left(\cdot \right) :[0,\pi ]\rightarrow \mathbb{R}%
^{+}$ is defined in \eqref{def:rho}.
\end{lemma}

\begin{proof}
From \eqref{Eq:Cov-T}, it is readily seen that
\begin{equation}
%\begin{split}
d_{T}^{2}(x,y) = \mathbb{E}\big(\left\vert T\left( x\right) -T\left(
y\right) \right\vert ^{2}\big) %=\mathbb{E}\left\vert T\left( x\right) \right\vert
%^{2}+\mathbb{E}\left\vert T\left( y\right) \right\vert ^{2}-2\mathbb{E}%
%T\left( x\right) T\left( y\right) \\
=\sum_{\ell =1}^{\infty }C_{\ell }\frac{2\ell +1}{2\pi }\big( 1-P_{\ell
}\left( \cos \theta \right) \big),
\end{equation}%
where we write for notational convenience $\theta =\theta_{xy}=
d_{\S^{2}}(x,y).$ Let
\begin{equation*}
Q_{\alpha }\left( \theta \right) =\sum_{\ell =1}^{\infty }\ell ^{-\alpha
}\Big( \ell +\frac{1}{2}\Big) \big( 1-P_{\ell }\left( \cos \theta
\right) \big).
\end{equation*}%
Schoenberg's theorem in  \cite{schoenberg1942} implies that, for every
 $\ell \ge 1$, $P_\ell(\langle x, y\rangle)$ is a covariance function on $\S^2$.
The Cauchy-Schwarz inequality gives $|P_{\ell }\left( \cos \theta
\right) |\le P_{\ell }\left( 1\right) =1$. Hence, it follows from Condition (A)
that
\begin{equation}
\frac{c_{0}^{-1}}{\pi }Q_{\alpha }\left( \theta \right) \leq d_{T
}^{2}(x,y)\leq \frac{c_{0}}{\pi }Q_{\alpha }\left( \theta \right).
\label{BoundMC}
\end{equation}%
The statement is then derived by exploiting Lemma \ref{SumPoly} %and \ref{P2}
in the Appendix, which provides a full characterization on the small scale
behaviour of $Q_\alpha \left( \theta \right) $ as
$\theta \rightarrow 0.$
\end{proof}

\begin{remark}
Anticipating some results to follow, it is important to stress the phase
transition that occurs in the behaviour of the canonical metric as a
function of $\alpha .$ For $\alpha > 4,$ the canonical metric is
proportional to the standard geodesic distance; for $2<\alpha <4,$ on the
contrary, the ratio between geodesic and canonical distance diverges on
small scales and fractal behaviour occurs. The case of $\alpha = 4$ is, in
some sense, critical and an extra logarithmic factor appears in the
bounds for the variogram in Lemma \ref{C1}.
\end{remark}

\subsection{The Construction of the Spherical Bump Function}

In this section, we work with spherical coordinates $(\vartheta ,\varphi ),$
$0\leq \vartheta \leq \pi ,$ $0\leq \varphi <2\pi ,$ and we review the
construction of a family of zonal functions $\delta _{\varepsilon}:
\S^{2}\rightarrow \mathbb{R},$ $\varepsilon >0,$ which shall vanish outside
a spherical cap around the North Pole $\vartheta =\varphi =0$ (we recall
that a zonal function satisfies by definition the identity $\delta
_{\varepsilon }(\vartheta ,\varphi )=\delta _{\varepsilon }(\vartheta
,\varphi ^{\prime })$ for all $\varphi $,$\varphi ^{\prime }\in \lbrack
0,2\pi )$)$.$ The construction follows a proposal by Geller and Mayeli
(\cite{gm1}, Lemma 4.1, pages 16-17), see also \cite{gm2}; we introduce
some minimal modifications, to ensure a suitable rate of decay in the
spherical harmonic coefficients. More precisely, we shall show that for
all $\varepsilon >0,$ there exists a zonal function
\begin{equation}
\delta _{\varepsilon }(\vartheta ,\varphi ):=\sum_{\ell =1}^{\infty }b_{\ell
}(\varepsilon )\frac{2\ell +1}{4\pi }P_{\ell }(\cos \vartheta )=\sum_{\ell
=1}^{\infty }\sum_{m=-\ell }^{\ell }\kappa _{\ell m}(\varepsilon )Y_{\ell
m}(\vartheta ,\varphi )  \label{prop0}
\end{equation}%
such that for some positive and finite constants $c_2$ and $c_3$,
we have
\begin{equation} \label{prop1}
\begin{split}
&\varepsilon ^{2}\delta _{\varepsilon }(\vartheta ,\varphi ) \leq  c_2 \  \ \hbox{
for all } \ 0\leq \vartheta \leq \pi,  \  0\leq \varphi <2\pi; \\
& \delta_{\varepsilon }(\vartheta ,\varphi )=0 \hbox{ for all  }\vartheta >\varepsilon
\end{split}
\end{equation}
and
\begin{equation} \label{prop2}
\delta_{\varepsilon }(0,0) \sim c_3 \varepsilon ^{-2} \ \hbox{ as } \varepsilon \to 0.
% \end{split}
\end{equation}%
Moreover the coefficients $\left\{ b_{\ell }(\varepsilon ),\kappa _{\ell
m} (\varepsilon )\right\} $ can be taken such that they satisfy
\begin{equation}
\begin{split}
&\left\vert b_{\ell }(\varepsilon )\right\vert \leq c_4, \quad \ \  \kappa _{\ell m}
(\varepsilon )= 0 \hbox{ for } m \ne 0, \hbox{ and }\\
&\left\vert \kappa _{\ell 0}(\varepsilon )\right\vert \leq
c_5\, \sqrt{2\ell +1}   \label{prop3}
\end{split}
\end{equation}
for all integers $\ell \ge 1$, where $c_4$ and $c_5$ are  positive and finite
constants.

It is natural to label $\delta_{\varepsilon }(\cdot,\cdot)$ a \emph{spherical
bump function,} in analogy with the analogous constructions on the
Euclidean domains.  On the other hand, up to a different normalization
factor the function $\delta_{\varepsilon }(\cdot, \cdot)$ is just a special case
of the so-called Mexican needlet frame by \cite{gm1}, in the special case
where the latter has bounded support in the real domain. We hence follow
as much as possible the notation by these authors.

In particular, we choose a function  $\widehat{G}(\cdot): \mathbb{R
\rightarrow R}$ such that it satisfies the following conditions:
\begin{itemize}
\item[(i).] $supp\widehat{G}(\cdot)\subseteq (-1,1)$,
\item[(ii).] It is piecewise continuously differentiable up to order $M$, where
$M$ is large enough, and
\item[(iii).] Its inverse Fourier transform $G$ satisfies $ 0 < \int_0^\infty G(u) u du
< \infty.$
\end{itemize}
For example, we can take  $\widehat{G}(\cdot) = p\star p (\cdot)$, where
$p(s) = \max\{0, 1 - 2|s|\}$. Then  $\widehat{G}(\cdot)$ is piecewise smooth
and its inverse Fourier transform is $G(u) = (\frac{2}{\pi})^2(1 -\cos(u/2))^2 u^{-4}$.
Functions $G(u)$ with faster decay rate of as $u \to \infty$ can be constructed
by convoluting more times.
% can be taken as the convolution of
%its inverse Fourier transform if the square of  a Polya distribution.

As in Geller and Mayeli \cite{gm1}, we consider the operator $G(\varepsilon
\sqrt{-\Delta_{\S^{2}}}):L^{2}(\S^{2})\rightarrow L^{2}(\S^{2})$ defined by%
\begin{equation*}
G(\varepsilon \sqrt{-\Delta _{\S^{2}}}):= \int_{-\infty }^{\infty }\widehat{G}%
(s)\exp (-is\varepsilon \sqrt{-\Delta _{\S^{2}}})\,ds;
\end{equation*}%
recall that $\Delta _{\S^{2}}$ is the spherical Laplacian in (\ref{Lap}).
The action of this operator is described as usual by means of the
corresponding kernel; i.e., for any $f\in L^{2}(\S^{2})$ we have%
\begin{equation*}
G(\varepsilon \sqrt{-\Delta_{\S^{2}}})f(\cdot):=\int_{\S^{2}}K_{\varepsilon
}(x,\cdot)f(x)\,dx,
\end{equation*}%
where%
\begin{equation}\label{Eq:K}
\begin{split}
K_{\varepsilon }(x,y) &:=\sum_{\ell =1}^{\infty }G \big(\varepsilon \sqrt{%
-\lambda _{\ell }} \big)\frac{2\ell +1}{4\pi }P_{\ell }(\left\langle
x,y\right\rangle ) \\
&=\sum_{\ell =1}^{\infty }\left\{ \int_{-\infty }^{\infty }\widehat{G}%
(s)\exp (-is\varepsilon \sqrt{-\lambda _{\ell }})ds\right\} \frac{2\ell +1}{%
4\pi }P_{\ell }(\left\langle x,y\right\rangle ).
\end{split}
\end{equation}%
In the above, $\left\{ \lambda _{\ell },\ \ell =1,2, ... \right\} $ are the
eigenvalues of  $\Delta _{\S^{2}}$,  i.e., $\lambda _{\ell }=-\ell (\ell +1)$,
\begin{equation*}
\Delta _{\S^{2}}Y_{\ell m}= \lambda _{\ell }Y_{\ell m}
\end{equation*}%
for $\ell = 1,2, \ldots$ and $m = -\ell, \ldots, \ell.$; see i.e, \cite{MPbook}, Chapter 3.

Under this assumptions, we take $x=N=(0,0)$ (the ``North Pole"), $%
y=(\vartheta ,\varphi )$ an arbitrary point on the sphere, and define
\begin{equation*}
\delta _{\varepsilon }(\vartheta ,\varphi ):= K_{\varepsilon }(N,y).
\end{equation*}%
Then the first inequality in (\ref{prop1})  follows from an application of Lemma 4.1
in \cite{gm1} to the case of  ${\mathbf M}=\S^2$ (hence  $n = 2$, $d(x, y) =
d_{\S^2}(N, y) =\vartheta$), $t = \varepsilon$ and $j, k, N = 0$. The second statement
in (\ref{prop1}), namely, $supp \delta_{\varepsilon } \subseteq \{(\vartheta, \varphi):
\vartheta \le \varepsilon\}$ follows from Huygens' principle as in the proof of Lemma
4.1 in  \cite[page 911]{gm1}.

To verify  (\ref{prop2}), we use the definition of $K$ in (\ref{Eq:K}) to verify that
as $\varepsilon \to 0$,
\[
\begin{split}
\delta _{\varepsilon }(0 ,0) &= \sum_{\ell =1}^\infty G\big(\varepsilon \sqrt{\ell(\ell + 1)}\big)
\frac{2 \ell + 1} {\sqrt{4\pi}} \\
&\sim \frac 1 {2 \pi}\int_0^\infty G(\varepsilon u) u du = c_3 \varepsilon^{-2},
\end{split}
\]
with $c_3 = (2 \pi)^{-1}\int_0^\infty G(u)u du$ which is positive  and finite.

Now we define
\begin{equation*}
b_{\ell }(\varepsilon ) :=\int_{-\infty }^{\infty }\widehat{G}(s)\exp
(-is\varepsilon \sqrt{\lambda _{\ell }})ds,
\end{equation*}
\begin{equation*}
\kappa _{\ell m}(\varepsilon ) =\left\{
\begin{array}{ll}
\sqrt{\frac{2\ell +1} {4\pi }}\, b_{\ell } (\varepsilon ), \  &\hbox{ if } m=0, \\
0 ,  &\hbox{ otherwise.}%
\end{array}%
\right.
\end{equation*}%
Then $|b_{\ell }(\varepsilon )| \le c$ for some constant $c$, and
$\{\kappa_{\ell m}(\varepsilon )\}$ satisfies the properties in \eqref{prop3}.
Moreover, by appealing to the standard identities%
\begin{equation*}
\frac{2\ell +1}{4\pi }P_{\ell }(\left\langle x,y)\right\rangle =\sum_{ m = \ell}^\ell
\overline{Y}_{\ell m}(x)Y_{\ell m}(y), \
\end{equation*}
\begin{equation*}
Y_{\ell m}(0,0) =\left\{ \begin{array}{cc}
\sqrt{\frac{2\ell +1}{4\pi }} , \  &\hbox{ for } m=0, \\
0, \  &\hbox{ otherwise,}%
\end{array}%
\right.
\end{equation*}%
we see that $ \delta _{\varepsilon }(\vartheta ,\varphi )$ can be written as
\begin{equation*}
\delta _{\varepsilon }(\vartheta ,\varphi )=\sum_{\ell =1}^{\infty }b_{\ell
}(\varepsilon )\frac{2\ell +1}{4\pi }P_{\ell }(\cos \vartheta )= \sum_{\ell=1}^\infty
\sum_{m= - \ell}^\ell \kappa _{\ell m}(\varepsilon )Y_{\ell m}(\vartheta ,\varphi ),
\end{equation*}%
which gives the desired representation in (\ref{prop0}).

We end this section with some further properties of the spherical bump function
$ \delta _{\varepsilon }(\vartheta ,\varphi )$ and its coefficient which will be used
in the proof of Theorem \ref{ND}  in Section \ref{LocalNondeterminism}.

To get information on the decay rate of $|b_{\ell }(\varepsilon )|$ as $\ell$
increases, we use  integration by parts $r$ times ($r \le M$) to get
\begin{equation*}
b_{\ell }(\varepsilon )=\int_{-\infty }^{\infty }\widehat{G}(s)\exp
(-is\varepsilon \sqrt{\lambda _{\ell }})ds  =\int_{-\infty }^{\infty }\widehat{ G}^{(r)}(s)
\frac{\exp (-is\varepsilon
\sqrt{\lambda _{\ell }})}{\left\{ i\varepsilon \sqrt{\lambda _{\ell }}\right\} ^{r}}ds.
\end{equation*}%
Hence for any $r \le M$,
\begin{equation}\label{prop4}
\big|b_\ell (\varepsilon) \big| \le  \frac{K_{r}} {\varepsilon^{r}\ell ^{r}},
\end{equation}
where
\begin{equation*}
K_{r}:=\sup_{-1\leq s\leq 1}\big\vert \widehat{G}^{(r)}(s)\big\vert < \infty.
\end{equation*}%
Note that, by  (\ref{prop2}), there exists a constant $\varepsilon_0 > 0$ such that
\begin{equation} \label{prop5}
\sum_{\ell =1}^{\infty }b_{\ell }(\varepsilon )\frac{2\ell +1}{4\pi }%
=\sum_{\ell=1 }^\infty \sum_{m= - \ell}^\ell \kappa _{\ell m}(\varepsilon )
\sqrt{\frac{2\ell +1}{4\pi }}%
=\delta _{\varepsilon }(0,0) \ge \frac{c_3} 2 \varepsilon ^{-2}
\end{equation}%
for all $\varepsilon \in (0, \varepsilon_0]$. Moreover, by \eqref{prop1}, we see that
for all $\vartheta >\varepsilon$,
\begin{equation} \label{prop6}
\begin{split}
\sum_{\ell =1}^{\infty }b_{\ell }(\varepsilon )\frac{2\ell +1}{4\pi }P_{\ell
}(\cos \vartheta ) &=\sum_{\ell m}\kappa _{\ell m}(\varepsilon )\sqrt{\frac{%
2\ell +1}{4\pi }}Y_{\ell m}(\vartheta ,\varphi ) \\
&=\delta _{\varepsilon }(\vartheta ,\varphi )=0.
\end{split}
\end{equation}%

%We remark that the function $\delta _{\varepsilon }(\cdot,\cdot)$ is oscillating,
%and indeed%
%\begin{equation*}
%\int_{\S^{2}}\delta _{\varepsilon }(\vartheta ,\varphi )\sin \vartheta d\vartheta
%d\varphi =\int_{\S^{2}}\sum_{\ell m}\kappa _{\ell m}(\varepsilon )\sqrt{\frac{%
%2\ell +1}{4\pi }}Y_{\ell m}(\vartheta ,\varphi )\sin \vartheta d\vartheta
%d\varphi =0.
%\end{equation*}%
%On the other hand, we also have that%
%\begin{eqnarray*}
%\int_{\S^{2}}\left\vert \delta _{\varepsilon }(\vartheta ,\varphi
%)\right\vert \sin \vartheta d\vartheta d\varphi &=&2\pi
%\int_{0}^{\varepsilon }\left\vert \delta _{\varepsilon }(\vartheta ,\varphi
%)\right\vert \sin \vartheta d\vartheta \leq \frac{2\pi c}{\varepsilon ^{2}}%
%\int_{0}^{\varepsilon }\sin \vartheta d\vartheta \\
%&=&\frac{2\pi c}{\varepsilon ^{2}}\left[ 1-\cos \varepsilon \right] \leq \pi
%c.
%\end{eqnarray*}%
%From this result, it also follows immediately that%
%\begin{equation*}
%\left\vert \kappa _{\ell 0}(\varepsilon )\right\vert \leq
%\int_{\S^{2}}\left\vert \delta _{\varepsilon }(\vartheta ,\varphi
%)\right\vert \left\vert Y_{\ell m}(\vartheta ,\varphi )\right\vert \sin
%\vartheta d\vartheta d\varphi \leq c\sqrt{2\ell +1}.  %\left\vert
%b_{\ell }(\varepsilon )\right\vert \leq c,
%\end{equation*}%

\section{Strong Local Nondeterminism: Proof of Theorem \ref{ND}
\label{LocalNondeterminism}}

We are now in the position to prove Theorem \ref{ND}. Recall that $T =
\{T(x), x \in \S^2\}$ is an isotropic Gaussian random field with mean zero
and angular power spectrum $\left\{ C_{\ell }\right\} .$ We prove the
following more general theorem which implies  Theorem \ref{ND} when
$2 < \alpha <4$. For $\alpha \ge 4$, the lower bound given by \eqref{Th:LND}
is strictly smaller than $\rho_\alpha^2( \varepsilon)$. Lemma \ref{C1}
indicates that \eqref{Th:LND} can be improved if $n=1$.
However, it is not known if one can strengthen \eqref{Th:LND} for all
$n \ge 2$.

\begin{theorem} \label{C2}
Under Condition (A), there exist positive and finite  constants  $\varepsilon_0$
and $c_{2}$ such that for all $\varepsilon \in (0, \varepsilon_0]$,  all integers
$n  \geq 1$ and all $x_{0},\, x_{1}, ...,\, x_{n}\in  \mathbb{S}^{2},$ satisfying
$ d_{\S^{2}}(x_{0},x_{k})\ge \varepsilon,$  we have
\begin{equation}\label{Th:LND}
{\rm Var}\left( T\left( x_{0}\right) |T\left( x_{1}\right) ,...,T\left(
x_{n}\right) \right) \geq c_{2}\varepsilon^{\alpha -2}.
\end{equation}
\end{theorem}

\begin{proof}
As before, we work in spherical coordinates $(\vartheta ,\varphi )$ and we
take without loss of generality $x_0=(0,0)$ to be the North Pole, and $%
x_{k}=(\vartheta _{k},\varphi _{k})$ so that $d_{\S^{2}}(x,x_{k})=\vartheta
_{k}.$ To establish \eqref{Th:LND}, it is sufficient to prove that there exists
a positive constant $c_2$ such that for all choices of real numbers
$\gamma _{1}, ..., \gamma _{n}$, we have
\begin{equation}\label{Eq:LND2}
\mathbb E\bigg\{ \bigg(T(0)-\sum_{j=1}^{n}\gamma _{j}T(x_{j})\bigg)^2\bigg\}
\ge c_2 \,\varepsilon ^{\alpha -2}.
\end{equation}
It follows from \eqref{Eq:T-rep1}, \eqref{Eq:Cov-T} or (\ref{Def:Cl}) that
\begin{equation*}
\begin{split}
\mathbb E\bigg\{ \bigg(T(0)-\sum_{j=1}^{n}\gamma _{j}T(x_{j})\bigg)^2 \bigg\}
&= \mathbb E\bigg\{ \bigg(
\sum_{\ell m}a_{\ell m}Y_{\ell m}(0)-\sum_{j=1}^{n}\gamma _{j}\sum_{\ell
m}a_{\ell m}Y_{\ell m}(x_{j}) \bigg)^2 \bigg\} \\
&=\sum_{\ell m}\mathbb E (\left\vert a_{\ell m}\right\vert ^{2}) \bigg\vert Y_{\ell
m}(0)-\sum_{j=1}^{n}\gamma _{j}Y_{\ell m}(x_{j})\bigg\vert ^{2} \\
&=\sum_{\ell } \sum_{m}C_{\ell }\bigg\vert Y_{\ell
m}(0)-\sum_{j=1}^{n}\gamma _{j}Y_{\ell m}(x_{j})\bigg\vert ^{2}.  \\
%&=\sum_{\ell } C_{\ell } \sum_{m}\bigg \vert Y_{\ell
%m}(x)-\sum_{j=1}^{n}\gamma _{j}Y_{\ell m}(x_{j})\bigg\vert ^{2}.
\end{split}
\end{equation*}%
Hence, (\ref{Eq:LND2}) is a consequence of Proposition \ref{Ing} below.
\end{proof}

\begin{proposition}
\label{Ing} Assume Condition (A) holds.  For all $\varepsilon \in (0, \varepsilon_0]$,
there exists a constant $c_2>0$ such that for all choices of $n\in \mathbb{N},$
all $(\vartheta _{j},\varphi _{j}): \vartheta_{j}>\varepsilon,$ and $\gamma
_{j}\in \mathbb{R},$ $j=1,2,...,n$, we have
\begin{equation}\label{Eq:LND3}
\sum_{\ell }\sum_{m} C_{\ell }\bigg[ Y_{\ell
m}(0,0)-\sum_{j=1}^{n}\gamma _{j}Y_{\ell m}(\vartheta _{j},\varphi _{j})\bigg]
^{2} \geq c_2\varepsilon ^{\alpha -2}.
\end{equation}
\end{proposition}

\begin{proof}
For any fixed $\varepsilon >0,$ let $\delta _{\varepsilon }(\cdot,\cdot)$ be defined as
in (\ref{prop0}), \ with the corresponding coefficients $\big\{b_{\ell m}(\varepsilon)
\big\}$ and $\big\{ \kappa _{\ell m}(\varepsilon )\big\}$ such that conditions (\ref{prop1}),
(\ref{prop2}), (\ref{prop3}), (\ref{prop4}), (\ref{prop5}) and (\ref{prop6}) hold. Now we
consider
\begin{equation*}
I = \sum_{\ell }\sum_{m}\left( \frac{\kappa _{\ell m}(\varepsilon )}{\sqrt{%
C_{\ell }}}\right) \left\{ \sqrt{C_{\ell }}\bigg[ Y_{\ell
m}(0,0)-\sum_{j=1}^{n}\gamma _{j}Y_{\ell m}(\vartheta_{j},\varphi_{j})\bigg] \right\}.
\end{equation*}%
On one hand, by the Cauchy-Schwartz inequality%
\begin{equation*}
\begin{split}
I^2 &\le %\left[ \sum_{\ell }\sum_{m}\left( \frac{\kappa _{\ell m}(\varepsilon )}{%
%\sqrt{C_{\ell }}}\right) \left\{ \sqrt{C_{\ell }}\left[ Y_{\ell
%m}(0,0)-\sum_{j=1}^{n}\gamma _{j}Y_{\ell m}(x_{j})\right] \right\} \right]^{2}
\left\{ \sum_{\ell m}\frac{\kappa _{\ell m}^{2}(\varepsilon )}{C_{\ell }}%
\right\} \Bigg\{\sum_{\ell }\sum_{m}  C_{\ell }\bigg[ Y_{\ell
m}(0,0)-\sum_{j=1}^{n}\gamma _{j}Y_{\ell m}(\vartheta_{j},\varphi_{j})\bigg] ^{2}\Bigg\}\\
& \leq \left\{ \sum_{\ell }\frac{(2\ell +1)}{4\pi }\frac{b_{\ell }^{2}(\varepsilon )%
}{C_{\ell }}\right\} \Bigg\{ \sum_{\ell }  C_{\ell }\sum_{m}\bigg[ Y_{\ell
m}(0,0)-\sum_{j=1}^{n}\gamma _{j}Y_{\ell m}(\vartheta_{j},\varphi_{j})\bigg]^{2}\Bigg\}.
\end{split}
\end{equation*}%
This inequality can be rewritten as
\begin{equation} \label{Eq:LND4}
\sum_{\ell }C_{\ell }\sum_{m}\bigg[ Y_{\ell
m}(0,0)-\sum_{j=1}^{n}\gamma _{j}Y_{\ell m}(\vartheta_{j},\varphi_{j})\bigg] ^{2}
\geq \frac{I^2}%\left\{ \sum_{\ell }\sum_{m}\kappa _{\ell m}(\varepsilon )\left[
%Y_{\ell m}(0,0)-\sum_{j=1}^{n}\gamma _{j}Y_{\ell m}(\vartheta_{j},\varphi_{j})\right] \right\}^{2}}
{\sum_{\ell }\frac{(2\ell +1)}{4\pi }\frac{b_{\ell
}^{2}(\varepsilon )}{C_{\ell }}  }.
\end{equation}%
On the other hand, we can compute $I^2$ directly. It follows from \eqref{prop5} and
\eqref{prop6} that
\begin{equation*}
\sum_{\ell }\sum_{m}\kappa _{\ell m}(\varepsilon )Y_{\ell m}(0,0) %=\sum_{\ell
%}\kappa _{\ell 0}(\varepsilon )Y_{\ell m}(0,0)
=\sum_{\ell }\frac{2\ell +1}{
4\pi }b_{\ell }(\varepsilon )=\delta _{\varepsilon }(0,0) \ge \frac{c_3}{ 2
\varepsilon ^{2}},
\end{equation*}%
and %for $x_j = (\vartheta_{j},\varphi_{j})$%
\begin{equation*}
\begin{split}
\sum_{\ell }\sum_{m}\kappa _{\ell m}(\varepsilon )\bigg\{
\sum_{j=1}^{n}\gamma _{j}Y_{\ell m}(\vartheta_{j},\varphi_{j})\bigg\} &=\sum_{j=1}^{n}\gamma
_{j}\sum_{\ell }\sum_{m}\kappa _{\ell m}(\varepsilon )Y_{\ell m}(\vartheta_{j},\varphi_{j}) \\
&=\sum_{j=1}^{n}\gamma _{j}\bigg\{ \sum_{\ell }\frac{2\ell +1}{4\pi }%
b_{\ell }(\varepsilon )P_{\ell }(\cos (N,x_{j}))\bigg\} \\
&=\sum_{j=1}^{n}\gamma _{j}\delta _{\varepsilon }(\vartheta _{j},\varphi
_{j})=0,
\end{split}
\end{equation*}%
because $\vartheta_{j}>\varepsilon $ by assumption.
The above two equations imply that  $I \ge  \frac{c_3}{ 2}  \varepsilon^{-2}$ and
hence \eqref{Eq:LND3} will follow from \eqref{Eq:LND4} if we can show that
\begin{equation}  \label{Eq:LND5}
\sum_{\ell }\frac{(2\ell +1)}{4\pi }\frac{b_{\ell }^{2}(\varepsilon )}{%
C_{\ell }}=O(\varepsilon ^{- \alpha +2}).
\end{equation}

Now we verify (\ref{Eq:LND5}). It follows from (\ref{prop4})  that for $r$ large
enough there exists a constant $c_{r}>0$ such that%
\begin{equation*}
b_{\ell }^{2}(\varepsilon )\leq \frac{c_{r}}{(\ell \varepsilon )^{r}}.
\end{equation*}%
Hence, by choosing an integer $L=L(\varepsilon )=\lfloor \varepsilon \rfloor
^{-1},$ we obtain
\begin{equation}\label{Eq:LND6}
\begin{split}
\sum_{\ell =1}^{\infty }\frac{(2\ell +1)}{4\pi }\frac{b_{\ell
}^{2}(\varepsilon )}{C_{\ell }} &=\sum_{\ell =L}^{\infty }\frac{(2\ell +1)}{%
4\pi }\frac{b_{\ell }^{2}(\varepsilon )}{C_{\ell }}+\sum_{\ell =1}^{L}\frac{%
(2\ell +1)}{4\pi }\frac{b_{\ell }^{2}(\varepsilon )}{C_{\ell }} \\
&\leq \frac{c_{r}}{\varepsilon ^{\alpha +2}}\sum_{\ell =L}^{\infty }(\ell
\varepsilon )\frac{1}{(\ell \varepsilon )^{r}}(\varepsilon \ell )^{\alpha
}\varepsilon +\sum_{\ell =1}^{L}\frac{(2\ell +1)}{4\pi }\frac{b_{\ell
}^{2}(\varepsilon )}{C_{\ell }}.
\end{split}
\end{equation}%
Now%
\begin{equation*}
\frac{c_{r}}{\varepsilon ^{\alpha +2}}\sum_{\ell =L}^{\infty }(\ell
\varepsilon )\frac{1}{(\ell \varepsilon )^{r}}(\varepsilon \ell )^{\alpha
}\varepsilon \leq \frac{c_{r}^{\prime }}{\varepsilon ^{\alpha +2}}%
\int_{1}^{\infty }x^{\alpha -r+1}dx\leq \frac{c_{r}^{\prime \prime }}{%
\varepsilon ^{\alpha +2}},
\end{equation*}%
for $r>\alpha +2,$ whereas we can  bound the second term from above by%
\begin{equation*}
\sum_{\ell=1 }^{L}\frac{(2\ell +1)}{4\pi }\frac{b_{\ell }^{2}(\varepsilon )}{%
C_{\ell }}\le c\sum_{\ell =1}^{L}\frac{(2\ell +1)}{4\pi }\ell ^{\alpha
}\le c\, L^{\alpha +2}\sim c\varepsilon ^{-(\alpha +2)}.
\end{equation*}%
Combining (\ref{Eq:LND6}) with the above verifies (\ref{Eq:LND5}), which finishes
the proof of (\ref{Eq:LND3}).
%Hence
%\begin{eqnarray*}
%\sum_{\ell }\left\{ C_{\ell }\sum_{m}\left[ Y_{\ell m}(0,0)-
%\sum_{j=1}^{n} \gamma _{j}Y_{\ell m}(x_{j})\right] ^{2}\right\}
%&\geq &\frac{\left\{ \sum_{\ell }\sum_{m}\kappa _{\ell m}(\varepsilon )\left[
%Y_{\ell m}(0,0)-\sum_{j=1}^{n}\gamma _{j}Y_{\ell m}(x_{j})\right] \right\}
%^{2}}{\left\{ \sum_{\ell }\frac{(2\ell +1)}{4\pi }\frac{b_{\ell
%}^{2}(\varepsilon )}{C_{\ell }}\right\} } \\
%&\geq &c\frac{\varepsilon ^{-4}}{\varepsilon ^{-(\alpha +2)}}\geq
%c\varepsilon ^{\alpha -2},
%\end{eqnarray*}%
%as claimed.
\end{proof}

\begin{remark}
At this stage we can draw an analogy between the isotropic spherical
random fields satisfying Condition (A) with $2< \alpha < 4$ and a fractional
Brownian field with self-similarity parameter $H.$ The analogy can be made
clearer by setting the parameter  values so that $2H+2=\alpha,$ and Lemma
\ref{C1} shows that the variogram of $T = \{T(x), x \in \S^2\}$ is of the order
$d_{\S^2}(x,y)^{2H}= d_{\S^2}(x,y)^{\alpha -2}.$ This indicates that $T $ shares
many analytic and fractal properties with a fractional Brownian field with
parameter $H.$ Indeed, by applying Lemma \ref{C1}  and Theorem \ref{ND},
we can prove that, for any $u \in \R$, the Hausdorff dimension of the level
set $T^{-1}(u)$ is given by
$$\dim_{\rm H} T^{-1}(u)= 2-\frac{\alpha -2}{2}, \quad \hbox{ a.s.}, $$
which shows that for $2<\alpha <4$ we have a fractal curve of Hausdorff
dimension $\in (1, 2)$.

Notice that, $\dim_{\rm H} T^{-1}(u)= 1$ when $\alpha \ge 4$, but the nature of
the level curve is different for  $\alpha > 4$ and $\alpha =4$, respectively.
For $\alpha > 4$,  the sample function $T(x)$ is differentiable. Thus its level
curve $T^{-1}(u)$ is regular. While for $\alpha = 4$ we believe that the level curve
is not differentiable and possesses subtle fractal properties. Investigation of the
topological and geometric properties of $T^{-1}(u)$ and more general excursion
sets in more details is left for future research.
\end{remark}

\section{Modulus of continuity: Proof of Theorem \ref{MC} \label{ModulusContinuity}}

We start by state 0-1 laws regarding the uniform and local moduli of continuity for
an isotropic spherical Gaussian field $T = \{T(x), x \in \S^2\}$. It is a consequence of
the representation (\ref{Eq:T-rep1}) and Kolmogorov's 0-1 law. %Similar 0-1 laws for
%Gaussian processes indexed by the Euclidean space can be found in
%\cite[Chapter 7]{MRbook}.
We first rewrite Lemma 7.1.1 in Marcus and Rosen \cite{MRbook} as follows.

\begin{lemma}\label{Lem-s4-1}
Let $\{T(x), x\in \S^2\}$ be a centered Gaussian random field on $\S^2$. Let
$\varphi: \R_+\to\R_+$ be a function with $\varphi(0+) = 0$. %Assume that there is
%a continuous map $\tau: \S^2 \mapsto\R_+$ with $\tau(0)=0$ such that  $d_G$ is
%continuous on $\tau$, i.e., $\tau(u_n-v_n)\to0$ implies
%$d_{G}(u_n,v_n)\to0$.
Then
$$
\lim_{\varepsilon \to 0}\sup_{^{\;\; x,y\in \S^2}_{d_{\S^2}(x,y)\leq\varepsilon} }
\frac{|T(x)-T(y)|}{\varphi (d_{\S^2}(x, y))}\leq K,\;\; \hbox{a.s.\, for some constant}\;\; K<\infty$$
implies that
 $$
 \lim_{\varepsilon \to0}\sup_{^{\;\; x,y\in \S^2}_{d_{\S^2}(x,y)\leq\varepsilon} }
 \frac{|T(x)-T(y)|}{\varphi(d_{\S^2}(x, y))}=K',\;\; \hbox{a.s.\, for some constant} \;\;
  K'<\infty.$$
%This result is also valid for the local modulus of continuity of $T$, that is, it
%holds with $y$ replaced by $x_0$ and with the supremum taken over $x \in \S^2$.
\end{lemma}

{\bf Proof of Theorem  \ref{MC}.} Because of Lemma \ref{Lem-s4-1}, we see
that \eqref{Eq:Umod} in Theorem  \ref{MC}  will be proved after we establish
upper and lower bounds of the following form:  If $2 < \alpha < 4$, then there
exist positive and finite constants $K_5$ and $K_6$ such that
\begin{equation}\label{Eq:Umod-upper}
\lim_{\varepsilon \rightarrow 0}\sup_{\substack{ x,y\in \S^{2}, \\ %
d(x,y)\le\varepsilon }}\frac{|T(x)-T(y)|}{d_{\S^{2}}(x,y)^{(\alpha -2)/2}\sqrt{%
\big| \log {d_{\S^{2}}(x,y)} \big|}} \le K_5, \ \ \hbox{ a.s. }
\end{equation}%
and
\begin{equation}\label{Eq:Umod-lower}
\lim_{\varepsilon \rightarrow 0}\sup_{\substack{ x,y\in \S^{2}, \\ %
d(x,y) \le \varepsilon }}\frac{|T(x)-T(y)|}{d_{\S^{2}} (x, y)^{(\alpha - 2)/2} \sqrt{ \big|
\log d_{\S^{2}} (x,y) \big|} } \ge K_6, \ \ \hbox{ a.s.}
\end{equation}%

We divide the rest of the  proof of Theorem  \ref{MC} into three parts.

{\it Step 1: Proof of \eqref{Eq:Umod-upper}.}  We introduce an auxiliary Gaussian field:
$$ Y=\{Y(x, y), \, x, y\in \S^2, d_{S^2}(x, y) \le \varepsilon\} $$
defined by $Y(x,y)=T(x)-T(y)$, where $\varepsilon > 0$ is
small so that \eqref{Eq:Vario-bounds} in Lemma \ref{C1} holds. The canonical metric
$d_Y$ on $\Gamma:= \{ (x, y)\in \S^2 \times \S^2: d_{\S^2}(x, y)
\le \varepsilon \}$ associated with $Y$ satisfies the following inequality:
\begin{equation}\label{Eq:dY}
d_Y((x,y),(x',y')) \leq  \min\{d_{T}(x,x')+ d_{T}(y,y'), d_{T}(x,y)+d_{T}(x',y')\}.
\end{equation}
Denote the diameter of $\Gamma$ in the metric $d_Y$ by $D$. Then, by \eqref{Eq:dY},
we have
$$ D\leq \sup_{(x, y)\in \Gamma} (d_{T}(x,y)+d_{T}(x',y')) \leq 2 \varepsilon.
$$
For any $\eta > 0$, let $N_Y(\Gamma,\eta)$ be the smallest number of open $d_Y$-balls of
radius $\eta$ needed to cover $\Gamma$. It follows from \eqref{Eq:dY} that for $2 < \alpha < 4$,
$$
N_Y(\Gamma,\eta) \leq K_7 \eta^{- \frac{4} {\alpha - 2} },
$$
for some positive and finite constant $K_7$, and one can verify that
$$
\int_0^D\sqrt{\log N_Y(T, \eta)} \, d\eta \leq K\, \varepsilon
\sqrt{\log(1+ \varepsilon^{-1})}. $$
Hence, by Theorem 1.3.5 in \cite{RFG}, we have
$$
\limsup_{\varepsilon \to 0}\, \sup_{^{\;\; x,y\in \S^2}_{d_{\S^2}(x,y)\leq \varepsilon} }
\frac{|T(x)-T(y)|} {\varepsilon^{(\alpha -2)/2} \sqrt{ |\log \varepsilon| }}\leq K, \quad \hbox{a.s.}
$$
for some finite constant $K$. One can verify (cf. Lemma 7.1.6 in \cite{MRbook})
that this implies \eqref{Eq:Umod-upper}.

{\it Step 2: Proof of \eqref{Eq:Umod-lower}.}  For any $n \geq \lfloor |\log_2
\varepsilon_0|\rfloor +1$, where $\varepsilon_0$ is as in Theorem \ref{C2}, we chose
a sequence of $2^n$ points $\{x_{n,i}, 1 \le i \le 2^n\}$ on $\S^2$ that are equally
separated in the following  sense: For every $2 \le k \le 2^n$, we have
\begin{equation}\label{Eq: xpoints}
 \min_{1 \le i \le k-1} d_{S^2}(x_{n,k}, x_{n,i}) = d_{S^2}(x_{n,k}, x_{n, k-1}) = 2^{-n}.
\end{equation}
There are many ways to choose such a sequence on $\S^2$.  Notice that
\begin{equation}\label{Eq:Umod-lower2}
\begin{split}
&\lim_{\varepsilon \rightarrow 0}\sup_{\substack{ x,y\in S^{2}, \\ %
d_{S^2}(x,y)\le \varepsilon } } \frac{|T(x)-T(y)|} {d_{S^{2}} (x, y)^{(\alpha - 2)/2}
\sqrt{ |\log d_{S^{2}}(x,y)|} } \\
& \ge \liminf_{n \to \infty} \max_{2 \le k \le 2^n}\frac{\big|T(x_{n, k})-T(x_{n, k-1} \big)| }
{2^{- n(\alpha - 2)/2} \sqrt{ n} }
\end{split}
\end{equation}%
It is sufficient to prove that, almost surely,  the last limit in (\ref{Eq:Umod-lower2}) is
bounded below by a positive constant. This is done by applying the property of strong
local nondeterminism in Theorem \ref{C2} and a standard Borel-Cantelli argument.

Let  $\eta > 0$ be a constant whose value will be chosen later. We consider
the events
$$
A_n = \Big\{ \max_{2 \le k \le 2^n}  \big |T(x_{n, k})-T(x_{n, k-1}) \big |
\le \eta 2^{- n(\alpha - 2)/2 }
\sqrt{ n} \Big\}
$$
and write
\begin{equation}\label{Eq:inter-1}
\begin{split}
\P \big(A_n\big) &= \P \Big\{ \max_{2 \le k \le 2^n-1}  \big |T(x_{n, k})-T(x_{n, k-1}) \big |
\le \eta 2^{- n(\alpha - 2)/2 } \sqrt{ n} \Big\}\\
&\qquad \times  \P \Big\{ \big |T(x_{n, 2^n})-T(x_{n, 2^n-1} ) \big |
\le \eta 2^{- n(\alpha - 2)/2} \sqrt{ n} \big | \widetilde{A}_{2^n-1} \Big\},
\end{split}
\end{equation}
where $\widetilde{A}_{2^n-1} = \big\{ \max_{2 \le k \le 2^n-1} \big |T(x_{n, k})-T(x_{n, k-1}) \big |
 \le \eta 2^{- n(\alpha - 2)/2 } \sqrt{ n} \big\}$.
The conditional distribution of the Gaussian random variable $T(x_{n, 2^n})-T(x_{n, 2^n-1} )$
under $\widetilde{A}_{2^n-1}$ is still Gaussian and, by Theorem  \ref{C2}, its conditional
variance satisfies
\[
\mathrm{Var}\big(T(x_{n, 2^n})-T(x_{n, 2^n-1}) \big| A_{n-1}\big)
\ge  c_2\, 2^{-(\alpha-2) n}.
\]
This and Anderson's inequality (see \cite{A55}) imply
\begin{equation}\label{Eq:inter-2}
\begin{split}
\P \Big\{ \big |T(x_{n, 2^n})-T(x_{n, 2^n-1} ) \big |
&\le \eta 2^{- n(\alpha - 2)/2} \sqrt{ n} \big | \tilde{A}_{2^n-1} \Big\}
 \le \P \Big\{ N(0, 1) \le c\, \eta \sqrt{n}\big\}\\
& \le  1 - \frac 1 { c\eta \sqrt{n}} \exp \Big(- \frac{c^2\eta^2 n}{2}\Big)\\
&\le \exp \bigg( - \frac{ 1} {c \eta \sqrt{n} } \exp \Big(- \frac{c^2\eta^2 n} {2}\Big)\bigg).
\end{split}
\end{equation}
In deriving the above, we have applied Mill's ratio and the elementary inequality
$1-x \le e^{-x} $ for $x > 0$. Iterating this procedure in (\ref{Eq:inter-1}) for $2^n$ times,
we obtain
\begin{equation}\label{Eq:inter-3}
\P \big(A_n\big) \le \exp \bigg( - \frac{ 1} {c \eta \sqrt{n} } 2^n\,
\exp \Big(- \frac{c^2\eta^2 n} {2}\Big)\bigg).
 \end{equation}
By taking $\eta > 0$ small enough such that $c^2 \eta^2 < 2 $, we have
$ \sum_{n=1}^\infty \P\big(A_n\big)$ $< \infty$. Hence the Borel-Cantelli lemma
implies that the right-hand side of (\ref{Eq:Umod-lower2}) is bounded from below
by $\eta > 0$.

{\it Step 3: Proof of \eqref{Eq:Umod2} for $\alpha =4$.} This is similar to the proof in Step 1,
except that the diameter $D$ of $\Gamma$ in the metric $d_Y$ is now comparable to
$K \varepsilon \sqrt {|\log \varepsilon|}$ and the covering number $
N_Y(\Gamma,\eta) \leq K \eta^{- 2}|\log \eta|$. Hence, in this case,
$$
\int_0^D\sqrt{\log N_Y(T, \eta)} \, d\eta \leq K\, \varepsilon |\log \varepsilon|. $$
Applying again Theorem 1.3.5 in \cite{RFG} yields that for $\alpha = 4$,
$$
\limsup_{\varepsilon \to 0}\, \sup_{^{\;\; x,y\in S^2}_{d_{S^2}(x,y)\leq \varepsilon} }
\frac{|T(x)-T(y)|} {\varepsilon  |\log \varepsilon| }\leq K, \quad \hbox{a.s.}
$$
Hence \eqref{Eq:Umod2} follows from this and Lemma 7.1.6 in \cite{MRbook}.
This finishes the proof of Theorem \ref{MC}.

%\emph{Here we need to establish exact upper and lower bound of the form}%
%\begin{equation*}
%\lim_{\varepsilon \rightarrow 0}\sup_{\substack{ x,y\in S^{2}, \\ %
%d(x,y)<\varepsilon }}\frac{|T(x)-T(y)|}{d_{S^{2}}(x,y)^{(\alpha -2)/2}\sqrt{%
%\ln \frac{1}{d_{S^{2}}(x,y)}}}=K>0, \ \hbox{ almost surely , for } 2<\alpha <4,
%\end{equation*}%
%\emph{whereas }%
%\begin{equation*}
%\lim_{\varepsilon \rightarrow 0}\sup_{\substack{ x,y\in S^{2}, \\ %
%d(x,y)<\varepsilon }}\frac{|T(x)-T(y)|}{d_{S^{2}}(x,y)\sqrt{\ln \frac{1}{%
%d_{S^{2}}(x,y)}}}=K>0, \ \hbox{ almost surely , for } 4<\alpha.
%\end{equation*}%

%\emph{I am not sure about the boundary case }$\alpha =4.$\emph{\ In any
%case, if I understood correctly from the papers by Yimin, the proof should
%have two steps: first we should show by metric entropy arguments that the
%limit is finite; is it then possible to show convergence by proving some
%monotonicity result? Then I guess that we will need to prove that }$K>0;$%
%\emph{\ is this the step where to use local nondeterminism and some
%Borel-Cantelli argument? For }$\alpha >4,$\emph{\ it seems to me the result
%should be trivial, because if I am correct the field is almost surely
%differentiable for these parameter values.}

\section{Higher-Order Derivatives: Proof of Theorem \ref{MC-2}
\label{HigherDerivatives}}

%Similarly to the approach pursued by \cite{lang1},\cite{lang2}, to study the
%regularity properties of higher-order derivatives we exploit
%pseudo-differential operators, as described in the classical monograph \cite%
%{Taylor}. In particular, for real $k\in \mathbb{R}$ introduce $(1-\Delta
%_{S^{2}})^{k/2}$ as the pseudo-differential operator whose action on
%functions $T\in L^{2}(S^{2})$ is defined by
%\begin{equation}\label{Tk2}
%\begin{split}
%(1-\Delta _{S^{2}})^{k/2}T &:=\sum_{\ell m}a_{\ell m}(T)(1+\ell (\ell
%+1))^{k/2}Y_{\ell m}, \\
%a_{\ell m}(T) &:=\int_{S^{2}}T(x)\overline{Y}_{\ell m}(x)dx.
%\end{split}
%\end{equation}%
%It is shown in \cite[Chapter XI ]{Taylor} that ....

Now we consider the case of $\alpha > 4$. Let $k\ge1$ be the integer such
that  $2 +2k < \alpha \le 4+2k$, and let $T^{(k)}= \{T^{(k)}(x), x \in \S^2\}$ be
the Gaussian random field defined by $ T^{(k)}=(1-\Delta _{\S^{2}})^{k/2}T.$
It follows from \eqref{Tk} that $T^{(k)}$ is again isotropic  and its angular
power spectrum is given by
\[
\widetilde{C}_{\ell }=\mathbb{E} \big(\left\vert a_{\ell m}\right\vert
^{2}\big) (1+\ell (\ell +1))^{k}=C_{\ell }(1+\ell (\ell +1))^{k}, \ \ \ \ell =1,2,...
\]
Under  Condition (A), we have $\widetilde{C}_{\ell } = \widetilde{G}
\left( \ell \right) \ell ^{2k-\alpha }$  for all $\ell =1,2,...$, where
\[
c_{6}^{-1}\leq \widetilde{G}\left( \ell \right) \leq c_{6}
\]%
for some finite  constant $c_6 \ge 1$. It follows from Theorem \ref{ND} that, for
all $n \ge 1$ and all $x_0, x_1, \ldots, x_n \in \S^2$ such that
$\min_{1 \le i \le n}d_{\S^2}(x_0, x_i)\le \varepsilon_0$, we have
\[
{\rm Var} \left( T^{(k)}\left( x_{0}\right) |T^{(k)}\left( x_{1}\right)
,...,T^{(k)}\left( x_{n}\right) \right) \geq c_{2}\min_{1 \le i \le n}d_{\S^2}(x_0, x_i)^{(\alpha
-2-2k)}.
\]%
Hence the conclusions of Theorem \ref{MC-2}  follow from Theorem \ref{MC}.
%that there exists a strictly positive constant $K>0$\ such that,
%with probability one%
%\[
%\lim_{\varepsilon \rightarrow 0}\sup_{\substack{ x,y\in S^{d}, \\ %
%d_{S^{2}}(x,y)<\varepsilon }}\frac{|T^{(k)}(x)-T^{(k)}(y)|}{\rho _{\alpha
%-2k}\left( d_{S^{2}}(x,y)\right) \sqrt{\left\vert \log \rho_{\alpha
%-2k}\left( d_{S^{2}}(x,y)\right) \right\vert }}=K.
%\]

%\emph{then we can easily show that, for }$2+k<\alpha <4+k$\emph{\ }%
%\begin{equation*}
%{\rm Var} \left( T^{(k)}\left( x_{0}\right) |T^{(k)}\left( x_{1}\right)
%,...,T^{(k)}\left( x_{n}\right) \right) \geq c_{2}\varepsilon ^{(\alpha
%-2-k)},
%\end{equation*}%
%\emph{and hence that there exists a strictly positive constant }$K>0$\emph{\
%such that, with probability one}%
%\begin{equation*}
%\lim_{\varepsilon \rightarrow 0}\sup_{\substack{ x,y\in S^{d}, \\ %
%d_{S^{2}}(x,y)<\varepsilon }}\frac{|T^{(k)}(x)-T^{(k)}(y)|}{\rho _{\alpha
%-k}\left( d_{S^{2}}(x,y)\right) \sqrt{\big|\log \rho _{\alpha -k}\left(
%d_{S^{2}}(x,y)\right) \big|}}=K.
%\end{equation*}%

%It follows from \eqref{Tk} that the spherical Gaussian field $T^{(k)}= \{T^{(k)}(x),
%x \in S^2\}$ is  isotropic and we can apply Theorem \ref{ND} to show that it has
%the property of strong local nondeterminism which, in turn, helps to establish
%the following exact uniform modulus of continuity of $T^{(k)}(x)$.
%\emph{It might be a bit tricky to justify this definition of regularity for
%the derivative terms, but it seems to me it is the only thing which is both
%doable and worth doing.@@@}

\section{Appendix \label{Appendix}}

In this Appendix we collect a number of technical results which are mainly
instrumental to investigate the behaviour of the canonical Gaussian metric
at small angular distances, in terms of the spectral index $\alpha .$

Let us first recall the Mehler-Dirichlet representation for the Legendre
polynomials (see \cite[eq. (13.9)]{MPbook} or \cite[Section 5.3, eq. (2)]{VMK}),
\begin{equation}\label{Eq:Pl}
P_{\ell }\left( \cos \vartheta \right) =\frac{\sqrt{2}}{\pi }\int_0^{\vartheta }%^\pi%
\frac{\cos \left( \left( \ell +\frac{1}{2}\right) \psi \right) }{\left( \cos
\psi -\cos \vartheta \right) ^{1/2}}d\psi,
\end{equation}
where the integral on the right hand side for $\vartheta= 0$ is understood
as the limit as $\vartheta \downarrow 0$.

In order to study the asymptotic behaviour of $\sum_{\ell =1}^{\infty }\ell
^{-s}P_{\ell }\left( \cos \vartheta \right)$ as $\vartheta \to 0$,  we  will
make use of the following identity: For any $ s>1,$
\begin{equation}\label{Eq:Pl2}
\sum_{\ell =1}^{\infty }\ell ^{-s}\cos \left( \Big( \ell +\frac{1}{2}%
\Big) \psi \right) =\func{Re}\bigg[ \sum_{\ell =1}^{+\infty }\ell
^{-s}e^{i\left( \ell +\frac{1}{2}\right) \psi }\bigg] =\func{Re}\left[ e^{%
\frac{i}{2}\psi }Li_{s}\left( e^{i\psi }\right) \right],
\end{equation}
where $Li_{s}\left( z\right) $ denotes the polylogarithm function, which is
defined as
\[
Li_{s}\left( z\right) :=\sum_{k=1}^{\infty }\frac{z^{k}}{k^{s}}
\]%
for $\left\vert z\right\vert <1$, and then extended holomorphically to
$\left\vert z\right\vert \geq 1.$

As usual, denote by $O\left( f(\cdot) \right) $ the terms that are no lower than the
order of $f(\cdot) $ and $o\left( f(\cdot)\right) $ having higher order than $
f(\cdot)$.  We have the following result:

\begin{lemma}
\label{SumPoly}\bigskip\ For any constant $s>1$,  as $\vartheta \to 0+$, we have%
\begin{equation*}
\sum_{\ell =1}^{+\infty }\ell ^{-s}P_{\ell }\left( \cos \vartheta \right)
=\left\{
\begin{array}{ll}
\zeta(s)-K_{7}\left( \sin {\vartheta } \right) ^{s-1}+
o\left( (\sin {\vartheta })^{s-1}\right) , &\hbox{ if } 1<s<3, \\
\zeta(s) -K_{8}\sin ^{2} \vartheta  \big|\ln \sin {\vartheta }\big| +O\left(
\sin^2 \vartheta  \right) , &\hbox{ if } s=3,
\\
\zeta(s)-K_{9}\sin ^{2} {\vartheta }  +O\left(
\sin^3 \vartheta \right)  , &\hbox{ if } s>3,%
\end{array}%
\right.
\end{equation*}%
where $\zeta \left( s\right) $  is the Riemann zeta function,  $K_{7},\, K_{8},
K_{9}$ are positive constants depending only on $s$.
\end{lemma}

\begin{proof}
We consider the two cases  $s\in
%TCIMACRO{\U{2115} }%
%BeginExpansion
\mathbb{N}
%EndExpansion
$ and  $s\notin
%TCIMACRO{\U{2115} }%
%BeginExpansion
\mathbb{N}
%EndExpansion
$, respectively.

{\it Case 1.} \, For $s\notin  \mathbb{N} ,$ we will exploit the series expansion
of $Li_{s}\left( e^{x}\right) $ for $x \in \mathbb C$ around the origin
 (see, \cite[eq. (9.4)]{Wood} or  \cite[Chapter 9]{GradRyzh}),
\begin{equation}
Li_{s}\left( e^{x}\right) =\Gamma \left( 1-s\right) \left( -x\right)
^{s-1}+\sum_{k=0}^{\infty }\frac{\zeta \left( s-k\right) }{k!}x^{k}.
\label{expLis}
\end{equation}%
Recall that the Riemann zeta function $\zeta \left( s\right) $ is well-defined
and holomorphic on the whole complex plane everywhere except for $s=1.$
The power series in (\ref{expLis}) converges in $\{x \in \mathbb C: |x| < 1\}$.

It follows that for $\vartheta >0$ small enough, and all $\psi \in (0, \vartheta)$,
\begin{equation}
\begin{split}
\func{Re}\big[ e^{\frac{i}{2}\psi }Li_{s}\left( e^{i\psi }\right) \big]
&=\cos \Big(\frac{\psi }{2}\Big)\left[ A_1   \psi ^{s-1}
  +\zeta \left( s\right) -\frac{%
1}{2}\zeta \left( s-2\right) \psi ^{2}   \right] \\
&\qquad + \sin \Big(\frac{\psi }{2}\Big) \Big[ B_1 \psi ^{s-1}\  -\zeta \left(
s-1\right) \psi +O \left(\psi ^{3}\right)\Big],  \label{ReLis}
\end{split}
\end{equation}%
where
\[
A_1 = \Gamma \left( 1-s\right)  \cos
\left( \frac{\pi }{2}\left( s-1\right) \right)\ \ \hbox{ and } B_1 =
\Gamma \left( 1-s\right) \sin
\left( \frac{\pi }{2}\left( s-1\right) \right)
\]
and we have incorporated $O\big(\psi^4\big)$ into
$O\big(\sin (\frac{\psi} 2 ) \psi^3 \big)$. Then, by (\ref{Eq:Pl}),
(\ref{Eq:Pl2}) and (\ref{ReLis})  above,  we have
\begin{equation}\label{6}
\begin{split}
&\sum_{\ell =1}^{\infty }\ell ^{-s}P_{\ell }\left( \cos \vartheta \right) \\%=%
%\frac{\sqrt{2}}{\pi }\int_{0}^{\vartheta }\frac{\func{Re}e^{\frac{i}{2}\psi
%} Li_{s}\left( e^{i\psi }\right) }{\left( \cos \psi -\cos \vartheta \right)
%^{1/2}}d\psi
%\end{equation*}%
%\begin{eqnarray}
&= \frac{\sqrt{2}}{\pi }  \int_{0}^{\vartheta } \frac{\cos \frac{\psi }{2}} {%
\left( \cos \psi -\cos \vartheta \right) ^{1/2}}\left[ A_1
 \psi ^{s-1} +\zeta \left( s\right) -\frac{1}{2}\zeta \left( s-2\right) \psi ^{2}\right]
d\psi  \\ %\label{4} \\
%&+\frac{\sqrt{2}}{\pi }\int_{0}^{\vartheta }\frac{\cos \frac{\psi }{2}}{%
%\left( \cos \psi -\cos \vartheta \right) ^{1/2}}\left[ \frac{1}{24}\zeta
%\left( s-4\right) \psi ^{4}+O\left( \psi ^{5}\right) \right] d\psi \\
%\label{5} \\
& \qquad + \frac{\sqrt{2}}{\pi } \int_{0}^{\vartheta }\frac{\sin \frac{\psi }{2}}{%
\left( \cos \psi -\cos \vartheta \right) ^{1/2}}\left[ B_1
  \psi ^{s-1}
-\zeta \left( s-1\right) \psi +O( \psi ^{3})\right] d\psi  \\
&:= J_1 + J_2.
\end{split}
\end{equation}
Recall that%
\begin{equation*}
\cos \psi -\cos \vartheta =2\sin ^{2}\frac{\vartheta }{2}-2\sin ^{2}\frac{%
\psi }{2}.
\end{equation*}%
A change of variable $x=\sin (\frac{\psi }{2})/\sin (\frac{\vartheta }{2})$ shows that for
$\gamma > 0$,
\begin{equation} \label{integ1}
\begin{split}
\int_{0}^{\vartheta } \frac{\sin ^{\gamma -1} \frac{\psi }{2}\, \cos \frac{\psi
}{2} } {\left( \cos \psi -\cos \vartheta \right) ^{1/2}} d \psi  &= \sqrt{2} \Big( \sin \frac{%
\vartheta }{2}\Big) ^{\gamma -1}\int_{0}^{1}\frac{x^{\gamma -1}}{\sqrt{
1-x^{2}}}dx  \\
&=\frac{\sqrt{2}}{2} B\Big( \frac{\gamma } {2},\frac{1}{2}\Big) \Big( \sin \frac{\vartheta }
{2}\Big) ^{\gamma-1} ,
\end{split}
\end{equation}%
and%
\begin{equation} \label{integ2}
\begin{split}
&\int_{0}^{\vartheta }\frac{\sin ^{\gamma -1}\frac{\psi }{2}}{\left( \cos
\psi -\cos \vartheta \right) ^{1/2}}d\psi  = \frac{\sqrt{2}}{2}\left( \sin \frac{\vartheta }
{2}\right) ^{\gamma -1}  \\
& \qquad \times \left[ B\Big( \frac{\gamma }{2},\frac{1}{2}\Big) +\frac{1}{6}
B\Big( \frac{\gamma }{2} +1,\frac{1}{2} \Big) \sin^2 \frac{\vartheta }{2}+
O\Big( \sin^{4}\frac{\vartheta }{2}\Big) \right].
\end{split}
\end{equation}%
By applying the following asymptotic expansion
%\[
%\frac{\psi}{\sin \psi} = 1 + \frac{(\sin \psi)^2} 6 + O\big(\sin^4 \psi\big),
%\]
%and
\[
\frac{\psi^{\beta}}{\sin^{\beta} \psi} = 1 + \beta \frac{(\sin \psi)^2} 6
+ O\big(\sin^4 \psi\big), \quad \hbox{ if } \ \beta > 0,
\]
we can use (\ref{integ1}) and (\ref{integ2}) to derive
\begin{equation}
\label{Eq:J1}
\begin{split}
J_1 &= A_2 \Big( \sin \frac{\vartheta }{2}\Big)^{s-1}  +\frac{1}{\pi }\zeta
\left( s\right) B\Big( \frac{1}{2},\frac{1}{2}\Big) \\
& \qquad-\frac{2}{\pi }\zeta \left( s-2\right) B\Big( \frac{3}{2},\frac{1}{2}\Big)
 \sin ^{2}\frac{\vartheta }{2}%
+O\bigg( \Big( \sin \frac{\vartheta }{2}\Big) ^{s+1}\bigg),
\end{split}
\end{equation}%
where $A_2$ is an explicit positive constant depending on $s$ only.
Likewise, we have
\begin{equation}
\label{Eq:J2}
\begin{split}
J_2  &= B_2 \sin^{s}\frac{\vartheta }{2}  -\frac{2}{\pi }\zeta \left( s-1\right)
B\Big( \frac{3}{2},\frac{1}{2}\Big) \sin ^{2}\frac{\vartheta }{2}%
 + O\Big( \sin ^{s+2}\frac{\vartheta }{2}\Big) ,
\end{split}%
\end{equation}
where $B_2$ is an explicit positive constant depending on $s$ only. By
combining (\ref{Eq:J1}) and (\ref{Eq:J2}), we derive that  for $s > 1$ and
$s\notin  \mathbb{N},$
\begin{equation*}\label{J12}
\begin{split}
\sum_{\ell =1}^{\infty }\ell ^{-s}P_{\ell }\left( \cos \vartheta \right)
&=\zeta(s)-C_{1}\Big( \sin \frac{\vartheta }{2}\Big) ^{s-1}  -
C_{2}\sin ^{2}\frac{\vartheta }{2} \\
& \qquad
+O\bigg( \Big(\sin \frac{\vartheta }{2}\Big)^{(s+1) \wedge 4}\bigg),
\end{split}%
\end{equation*}
where $C_{1}$ and $C_{2}$ are positive constants depending
only on $s$, and  $a\wedge b = \min\{a, b\}$. Consequently,
\begin{equation} \label{Less3}
\sum_{\ell =1}^{\infty }\ell ^{-s} P_{\ell }\left( \cos \vartheta \right)
=\zeta(s) - C_{1}\Big( \sin \frac{\vartheta }{2}\Big) ^{s-1}+O\Big(
\sin^2 \frac{\vartheta }{2}  \Big)
\end{equation}%
for $1<s<3,\ s\neq 2$, and
\begin{equation} \label{Great3}
\sum_{\ell =1}^{\infty }\ell ^{-s} P_{\ell }\left( \cos \vartheta \right)
=\zeta(s) - C_{2} \sin^2 \frac{\vartheta }{2} +O\Big( \sin^4
\frac{\vartheta }{2}\Big) ,
\end{equation}%
for $ s>3, \, s \notin \mathbb N$.

{\it Case 2}. \ If $s > 1$ and $s = n \in \mathbb{N},$ we make use
of the following series expansion of $Li_{n}\left(e^{x}\right) $ %about $x=0,$
(see  \cite[eq. (9.5)]{Wood} or \cite[Chapter 9]{GradRyzh}) for $x \in
\mathbb C$ with $ |x| < 1$,
\begin{equation}  \label{expLiN}
Li_{n}\left( e^{x}\right) =\frac{x^{n-1}}{\left( n-1\right) !}\left[
H_{n-1}-\ln \left( -x\right) \right] +\sum_{k=0,k\neq n-1}^{\infty }\frac{%
\zeta \left( n-k\right) }{k!}x^{k},
\end{equation}%
where $H_{n}$ denotes the $n$-th harmonic number:%
\begin{equation*}
H_{n}=\sum_{j=1}^{n}\frac{1}{j},\ \ \ H_0 = 0.
\end{equation*}%
%Hence
%\begin{equation*}
%Li_{n}\left( e^{x}\right) =\frac{x^{n-1}}{\left( n-1\right) !}\left[
%H_{n-1}-\ln \left( -x\right) \right] +\sum_{k=0}^{n-2}\frac{\zeta \left(
%n-k\right) }{k!}x^{k}+O\left( x^{n}\right) .
%\end{equation*}%
It follows that%
\begin{equation*}
\begin{split}
\func{Re}\left[ e^{\frac{i}{2}\psi }Li_{n}\left( e^{i\psi }\right) \right]
&=\func{Re}\left[ e^{\frac{i}{2}\psi }\frac{i^{n-1}\psi ^{n-1}}{\left(
n-1\right) !}\Big( H_{n-1}-\ln \psi +\frac{\pi }{2}i\Big) \right]  \\
& \qquad +\func{Re}\bigg[ \sum_{k=0,k\neq n-1}^{n+1}\frac{\zeta \left( n-k\right) }{%
k!}i^{k}\psi ^{k}\bigg] +O\left( \psi ^{n+2}\right).
\end{split}
\end{equation*}%

If $n$ is an odd integer, then
\begin{equation*}\label{odd}
\begin{split}
\func{Re}\left[ e^{\frac{i}{2}\psi }Li_{n}\left( e^{i\psi }\right) \right]
&=\left(-1\right) ^{\left( n-1\right) /2}\frac{\psi ^{n-1}}{\left(
n-1\right) !}
\left[ \big( H_{n-1}-\ln \psi \big) \cos \frac{\psi }{2} -%
\frac{\pi }{2}\sin \frac{\psi }{2}\right]  \\
&\qquad +\sum_{k=0,k\neq \left( n-1\right) /2}^{\left( n+1\right) /2}\frac{\zeta
\left( n-2k\right) }{k!}\left( -1\right) ^{k}\psi ^{2k}+O\left( \psi
^{n+3}\right) .
\end{split}
\end{equation*}%
Thus, one can see that%
\begin{equation} \label{IntPsi}
\begin{split}
&\sum_{\ell =1}^{\infty }\ell ^{-n}P_{\ell }\left( \cos \vartheta \right)
= \frac{\sqrt{2}}{\pi }\frac{\left( -1\right) ^{\left( n-1\right) /2}}{%
\left( n-1\right) !}\\
&\qquad \times \int_{0}^{\vartheta }\frac{\psi ^{n-1}}{\left( \cos \psi
-\cos \vartheta \right)^{1/2}}\left[\big( H_{n-1}-\ln
\psi \big)  \cos \frac{\psi }{2} -\frac{\pi }{2}\sin \frac{\psi }{2}\right] d\psi  \\
&\qquad +\frac{\sqrt{2}}{\pi }\sum_{k=0,k\neq \left( n-1\right) /2}^{\left(
n+1\right) /2}\frac{\zeta \left( n-2k\right) }{k!}\left( -1\right)
^{k}\int_{0}^{\vartheta }\frac{\psi ^{2k}}{\left( \cos \psi -\cos \vartheta
\right) ^{1/2}}d\psi \\
&\qquad +O\bigg( \int_{0}^{\vartheta }\frac{\psi ^{n+3}}{\left(
\cos \psi -\cos \vartheta \right)^{1/2}}d\psi \bigg).
\end{split}%
\end{equation}

Observe that, in (\ref{IntPsi}), the term corresponding to $k =0$
goes to $\zeta(n)$ as $\vartheta \to 0+$, and the leading integral is
\[
J_3=\int_{0}^{\vartheta }\frac{\psi ^{n-1}  \ln \psi  }{\left( \cos \psi -\cos \vartheta
\right)^{1/2}}\cos \Big(\frac{\psi }{2}\Big)d\psi.
\]
By a change of variable $y=\sin ^{2}\frac{\psi }{2}/\sin^{2}\frac{\vartheta }{2},$
we can write $J_3$ as
\begin{equation*}  \label{3}
\begin{split}
J_3& = \frac{2^{n-1}}{\sqrt{2}}\sin ^{n-1}\frac{\vartheta }{2}\int_{0}^{1}\frac{%
y^{\frac{n}{2}-1}\left( 1+\sin ^{2}\vartheta \frac{n-1}{6}y+O\left( \sin
^{4}\vartheta y^{2}\right) \right) }{\left( 1-y\right) ^{1/2}}   \\
&\qquad \times \left( \ln y+2\left( \ln \sin \vartheta +\ln 2\right) +\frac{\sin
^{2}\vartheta }{6}y+O\left( \sin ^{4}\vartheta y^{2}\right) \right) dy.
%&:= J_3 + J_4.
\end{split}
\end{equation*}%
For $n\geq 3,$ we derive
\begin{equation}
\begin{split}\label{J3}
J_3 &=\frac{2^{n-1}} {\sqrt{2}} (1 + 2\ln 2) B_{\ln } \Big( \frac{n}{2},
\frac{1}{2}\Big) \sin^{n-1}\vartheta\\
&\qquad +\frac{2^{n}}{\sqrt{2}}  B\Big( \frac{n}{2},\frac{1}{2}\Big)
\sin ^{n-1}\vartheta \cdot \Big( \ln \sin \frac{%
\vartheta }{2}\Big)   +O\big(\sin^{4}\vartheta \big),
\end{split}
\end{equation}
where%
\begin{equation*}
B_{\ln }\left( a,b\right) =\int_{0}^{1}\frac{x^{a-1}\ln x}{\left( 1-x\right)
^{1-b}}dx=-\int_{0}^{1}B\left( y;a,b\right) \frac{1}{y}dy,
\end{equation*}%
and $B\left( y;a,b\right) $ is the so-called incomplete Beta function,
defined as%
\begin{equation*}
B\left( y;a,b\right) =\int_{0}^{y}\frac{x^{a-1}}{\left( 1-x\right) ^{1-b}}dx.
\end{equation*}%
By combining (\ref{IntPsi}) and (\ref{J3})  we see that, if $s=n >1$
is an odd integer,  then
\begin{equation} \label{SumP2}
\begin{split}
\sum_{\ell =1}^{\infty }\ell^{-n}P_{\ell }\left( \cos \vartheta \right)
&=\zeta(n) - D_{1}\sin ^{2}\frac{\vartheta }{2} + \delta _{n}^{3} D_{2 }
\sin ^{2}\frac{\vartheta }{%
2} \cdot \big( \ln \sin \frac{\vartheta }{2}\big)\\
&\qquad \quad + O\Big( \sin^{3} \frac{\vartheta }{2}\Big) ,
\end{split}
\end{equation}%
where $\delta _{i}^{j}=1$ if $i=j$ and $0$ otherwise,  $ D_{1}$
and $D_{2} $ are positive constants depending on $s$ only. Consequently,
if $s > 1$ is an odd integer, then
\begin{equation}
\sum_{\ell =1}^{\infty }\ell ^{-n}P_{\ell }\left( \cos \vartheta \right)
=\left\{
\begin{array}{ll}
\zeta(n)+D_{2 } \sin ^{2}\frac{\vartheta }{2} \big(\ln \sin \frac{\vartheta }{2}\big)%
+O\left( \sin ^{2}\frac{\vartheta }{2}\right) , & \hbox{ if }\, s=3, \\
\ \\
\zeta(n) -D_{1}\sin ^{2}\frac{\vartheta }{2}+O\left( \sin^{3}\frac{\vartheta }{2%
}\right) , & \hbox{ if }\, s\geq 5.%
\end{array}%
\right.   \label{Odd}
\end{equation}

Finally, we consider the case when $s =n >1$ is an even integer. It follows from
\eqref{expLiN} that
%Otherwise, if $n$ is even, then%
\begin{equation}
\begin{split}
\func{Re}\left[ e^{\frac{i}{2}\psi }Li_{s}\left( e^{i\psi }\right) \right]
&=\left( -1\right) ^{n/2}\frac{\psi ^{n-1}}{\left( n-1\right) !}\left[
\big(H_{n-1}-\ln \psi \big) \sin \frac{\psi }{2}  +\frac{\pi }{2}\cos \frac{%
\psi }{2}\right]  \\
&\qquad +\sum_{k=0}^{n/2+1}\frac{\zeta \left( n-2k\right) }{k!}\left( -1\right)
^{k}\psi ^{2k}+O\left( \psi ^{n+4}\right),
\end{split}%
\end{equation}
which leads to
\begin{equation}
\begin{split}
&\sum_{\ell =1}^{\infty }\ell ^{-n}P_{\ell }\left( \cos \vartheta \right)   %=
%\frac{\sqrt{2}}{\pi }\int_{0}^{\vartheta }\frac{\func{Re}e^{\frac{i}{2}\psi
%}Li_{n}\left( e^{i\psi }\right) }{\left( \cos \psi -\cos \vartheta \right)
%^{1/2}}d\psi
=\frac{\sqrt{2}}{\pi }\frac{\left( -1\right) ^{n/2}}{\left( n-1\right) !}\\
& \qquad \times \int_{0}^{\vartheta }\frac{\psi ^{n-1}}{\left( \cos \psi -\cos \vartheta
\right) ^{1/2}}\left[\big( H_{n-1}-\ln \psi \big) \sin \frac{\psi }{2} -
\frac{\pi }{2}\cos \psi \right] d\psi  \\
& \qquad +\frac{\sqrt{2}}{\pi }\sum_{k=0}^{n/2+1}\frac{\zeta \left( n-2k\right) }{k!%
}\left( -1\right) ^{k}\int_{0}^{\vartheta }\frac{\psi ^{2k}}{\left( \cos
\psi -\cos \vartheta \right) ^{1/2}}d\psi\\
&\qquad  +O\bigg( \int_{0}^{\vartheta }%
\frac{\psi ^{n+4}}{\left( \cos \psi -\cos \vartheta \right) ^{1/2}}d\psi \bigg).
\end{split}%
\end{equation}
Similarly to the case when $s$ is odd,  we can derive that for $s=n$ even,
\begin{equation} \label{SumP3}
\begin{split}
\sum_{\ell =1}^{\infty } \ell ^{-n}P_{\ell }\left( \cos \vartheta \right)
&=\zeta(n)- \delta _{n}^{2}\bigg\{ B\Big( 1,\frac{1}{2}\Big) \sin \frac{\vartheta }{2}
- D_{3}\sin ^{2}\frac{\vartheta }{2}\cdot\big( \ln \sin \frac{\vartheta }{2}\big) \bigg\} \\
&\qquad -D_{4}\sin ^{2}\frac{\vartheta }{2} +O\Big( \sin ^{3}\frac{\vartheta }{2}\Big),
\end{split}%
\end{equation}
where $D_{3}$ and $ D_{4 } $ are positive constants
depending on $s$ only. That is, for even integer $s>1,$  we have
\begin{equation} \label{Even}
\sum_{\ell =1}^{\infty }\ell ^{-s}P_{\ell }\left( \cos \vartheta \right)
=\left\{
\begin{array}{ll}
\zeta(s) - 2 \sin \frac{\vartheta }{2}+o\left( \sin
\frac{\vartheta }{2}\right) , &\hbox{ if }\, s=2, \\
\ \\
\zeta(s) +D_{2}\sin ^{2}\frac{\vartheta }{2}+O\left( \sin ^{3}\frac{\vartheta }{2%
}\right) , & \hbox{ if }\, s\geq 4.%
\end{array}%
\right.
\end{equation}%
This completes the proof of Lemma \ref{SumPoly} in view of
( \ref{Less3}),  (\ref{Great3}), (\ref{Odd}) and (\ref{Even}).
\end{proof}

\bigskip

\bigskip

{\bf Acknowledgement} Research of X. Lan is supported by NSFC grant 11501538;
D. Marinucci is supported by ERC Grant n. 277742 \emph{Pascal};
Y. Xiao is partially supported by NSF grants DMS-1307470 and DMS-1309856.

\newpage

\noindent Xiaohong Lan \\
School of Mathematical Sciences, University of Science and Technology of China\\
Hefei, 230061, People's Republic of China\\
email: xhlan@ustc.edu.cn

\medskip

\noindent Domenico Marinucci \\
Department of Mathematics, University of Rome Tor Vergata\\
via dell Ricerca Scientifica 1, 00133, Roma, Italy\\
email: marinucc@mat.uniroma2.it

\medskip

\noindent Yimin Xiao \\
Department of Statistics and Probability, Michigan State University \\
East Lansing, Michigan, 488224, USA\\
email: xiao@stt.msu.edu

\end{document}